\documentclass[11pt,reqno]{amsart}
\usepackage[shortlabels]{enumitem}
\usepackage{esint}
\usepackage{physics}
\usepackage{relsize}
\usepackage[backref]{hyperref}
\usepackage{mathtools}
\usepackage{amsfonts}
\usepackage[all]{xy}
\usepackage{geometry}
\usepackage{amsmath,amssymb} 
\usepackage{faktor}
\usepackage{amscd}
\usepackage{tikz-cd}
\usepackage{booktabs,tabularx,array}
\usepackage{capt-of}

\newcommand{\subsectionnotoc}[1]{%
  \par\addvspace{.5\baselineskip plus .7\baselineskip}%
  \noindent{\normalfont\bfseries #1.}\enspace\ignorespaces
}

\newtheorem{theorem}{Theorem}[section]
\newtheorem{lemma}[theorem]{Lemma}
\newtheorem{definition}[theorem]{Definition}

\theoremstyle{corollary}
\newtheorem{corollary}[theorem]{Corollary}
\theoremstyle{conjecture}

\theoremstyle{assumption}

\theoremstyle{proposition}
\newtheorem{proposition}[theorem]{Proposition}
\theoremstyle{remark}
\newtheorem{remark}[theorem]{Remark}
\numberwithin{equation}{section}
\everymath{\displaystyle}

\newcommand{\ii}{\ensuremath{\sqrt{-1}}}
\newcommand{\pp}{\bar\partial}
\newcommand{\C}{\mathbb{C}}
\newcommand{\noop}[1]{}

\DeclareMathOperator{\Ric}{Ric}
\newcommand{\Aut}{\operatorname{Aut}}
\DeclareMathOperator{\Hess}{Hess}

\title{On compact steady pluriclosed soliton surfaces}
\author{Junsheng Zhang}
\address{Courant Institute of Mathematical Sciences\\
  New York University, 251 Mercer St\\
  New York, NY 10012\\}
\curraddr{}
\email{jz7561@nyu.edu}
\author{Keshu Zhou}
\address{Department of Mathematics, University of California, Berkeley, CA 94720, USA} 
\email{keshu\_zhou@berkeley.edu}

\date{}

\begin{document}

\begin{abstract}
    We complete the classification of non-K\"ahler steady pluriclosed soliton on compact complex surfaces, as initiated by Streets. The underlying complex structure must be a  Hopf surface, with any finite primary cover being of class $1$. Moreover, the soliton metric is unique up to biholomorphism.
\end{abstract}

\maketitle

\section{Introduction}

Let $(M,J)$ be a compact connected complex surface. We say $(g,f)$ is a steady (gradient) pluriclosed soliton if $g$ is a pluriclosed metric on $(M,J)$ and $f\in C^{\infty}(M,\mathbb R)$ such that the following hold:
\begin{equation}
    \begin{aligned}
        &\Ric_g+\nabla^2 f=\frac{1}{4}H^2, \\
        &d^*H+\iota_{\nabla f}H=0.
    \end{aligned}
\end{equation}
Here $H=-d^c\omega$ is the closed Bismut torsion $3$-form, and $H^2=H_{ipq}H^{pq}_j$ is a $(0,2)$-tensor. This type of equation emerges naturally from the study of pluriclosed flows and the corresponding geometrization program for compact complex surfaces; we refer the readers to \cite{streets20} for a detailed explanation. Below we will instead restrict our attention to the steady pluriclosed soliton itself.

In \cite[Theorem 1.1]{Streets19}, it was proved that any steady soliton must either be a K\"ahler Ricci-flat metric, or the underlying surface must be biholomorphic to a Hopf surface. Moreover, Streets constructed a steady soliton metric on every Hopf surface admitting a finite primary cover of class $1$; see terminology in Section \ref{subsec--Hopf surfaces}. The remaining question is whether such a construction exhausts all compact steady soliton surfaces \cite[Remark 1.2(3)]{Streets19}. In this paper, we answer this question in the following two theorems.

\

Our first result is the classification of complex structures:
\begin{theorem}\label{thm--classification of complex structure}
     Suppose $(M,J)$ is non-K\"ahler and admits a steady pluriclosed soliton. Then any finite primary cover of $(M,J)$ is a Hopf surface of class $1$.
\end{theorem}

We remark that in \cite{Streets19}, the fact that $(M,J)$ is a Hopf surface was obtained by using the classification of class VII surface with $b_2=0$, commonly known as Bogomolov's theorem \cite{Bogomolov76,LYZ90,Teleman94}. In the proof below, we will prove Theorem \ref{thm--classification of complex structure} without appealing to this classification theorem.

The second theorem addresses the uniqueness question of soliton metric on class $1$ Hopf surfaces. Let $(M,J)$ be a Hopf surface of class $1$, and denote by $(g_0,f_0)$ the steady soliton constructed in \cite{Streets19}, with an alternative construction in \cite{Streets-Ustinovskiy21}. Then we have:
\begin{theorem}\label{thm--uniqueness of soliton metric}
    Suppose $(g,f)$ is a steady pluriclosed soliton on $(M,J)$. Then there exists a biholomorphism $\Phi:(M,J)\to(M,J)$ and $\lambda>0$, such that $\Phi^*g=\lambda g_0$ and $\Phi^*f=f_0+\mathrm{const}$.
\end{theorem}

As pointed out to us by Streets, the results developed in his work  suggest that every steady pluriclosed soliton on a compact complex surface should be generalized K\"ahler \cite{GualtieriGK}. This is confirmed as a consequence of the above two results together with \cite[Theorem 1.1]{Streets-Ustinovskiy21}.

\

We briefly explain the idea of proof. The key new observation here is that the soliton equation itself implies the non-negativity of Bakry--\'Emery Ricci curvature\footnote{This was also observed in the very recent work of Streets \cite{Streets26}, applied more generally for generalized Ricci solitons.}, thus together with compactness of $M$ and restriction on $b_1(M)$, automatically produces the sufficient symmetries via Bochner techniques. The existence of real holomorphic Killing fields then puts a strong restriction on the underlying complex structure, thus ruling out the class~$0$ Hopf surface. Meanwhile, the symmetry naturally reduces the soliton equation to ODEs used in \cite{Streets-Ustinovskiy21}, from which the uniqueness of soliton metric follows.  This proof is similar in spirit to that in \cite{apostolov2023}.

\subsectionnotoc{Notation and Conventions} 
\begin{itemize}
    
\item $d^c\coloneqq\ii(\pp-\partial)$ and $\omega(\cdot,\cdot)=g(J\cdot,\cdot)$
\item The Lee form $\theta\coloneqq -(d^*\omega)\circ J$. In particular, $H=-*\theta$ in complex dimension $2$.
\item We use $\nabla^g$ to denote the Levi-Civita connection of the Riemannian metric $g$ and $d\mu_g$ to denote the volume measure associated with a Riemannian metric $g$.
\item We use $\delta$ and $d^*$ interchangeably to denote the formal adjoint of $d$ with respect to the Riemannian measure $d\mu_g$.
\item We use $\nabla^B$ to denote the Bismut connection, defined as $\nabla^B=\nabla^g-\frac12 g^{-1}d^c\omega$. The Bismut Ricci form is denoted by $\rho_B$, defined as 
\[
\rho_B(X,Y)=\frac12\langle R^B(X,Y)Je_i,e_i\rangle,
\]
where $R^B$ is the Riemann curvature tensor for Bismut connection and $\{e_i\}$ is any orthonormal basis on $T_pM$ at a given point.
\end{itemize}
 
\subsectionnotoc{Acknowledgments} 
The work grew out of a reading seminar on compact complex surfaces organized by
 Song Sun at the Institute for Advanced Study in Mathematics (IASM),
Zhejiang University, during the summer of 2026. The authors thank Song Sun for
valuable discussions that led to this work, and Jeffrey Streets for fruitful discussion and helpful comments. They also thank IASM for its
hospitality, and Yifan Chen, Mingyang Li, Chunhui Wei, Qi Yao, Yichen Zhang,
and Yuanbo Zhou for their participation in the seminar and for stimulating
discussions.

\subsectionnotoc{Declaration of AI usage} ChatGPT Pro with GPT-5.6  are used as a tool for literature searches, and for exploratory discussions of ideas and proof strategies proposed by the human authors. The author takes full responsibility for the paper's content
and correctness.

\section{Preliminaries}

\subsection{Bakry--\'Emery theory}\label{sub--BE theory}
We refer to \cite{Lott03,WW09} for the Bochner techniques in the setting of Bakry--\'Emery Ricci
curvature. We use the standard notation
\begin{equation}
     \Ric_f := \Ric_g + \Hess^g f,
\end{equation}
   which is also referred to as the $\infty$--Bakry--\'Emery Ricci curvature. The target metric measure space is the weighted space $(M^n, g ,e^{-f}d\mu_g)$. 

Let $\delta$ be the formal adjoint of $d$ with respect to unweighted Riemannian measure $d\mu_g$, and $\delta_{f}$ be the formal adjoint with respect to $e^{-f}d\mu_g$. Then we have
\begin{equation}
    \delta_f=\delta+\iota_{\nabla f}.
\end{equation}
One can then define the weighted Hodge Laplacian operator
\begin{equation}
    \Delta_f\coloneqq d\delta_f+\delta_fd
\end{equation}
acting on space of $k$-forms. The weighted Bochner formula for $1$-forms then turns out to be:
\begin{equation}
    \frac12\Delta_f(\vert\eta\vert^2)=\langle\eta,\Delta_f\eta\rangle-\vert\nabla\eta\vert^2-\Ric_f(\eta^{\#},\eta^{\#}).
\end{equation}
For the proof, we refer to \cite[Section 2]{Lott03}. Integrate over $M^n$ against the weighted measure, we then obtain:
\begin{equation}
    \int_M\langle\eta,\Delta_f\eta\rangle e^{-f}d\mu_g=\int_M\big(\vert\nabla\eta\vert^2+\Ric_f(\eta^{\#},\eta^{\#})\big)e^{-f}d\mu_g.
\end{equation}
Thus if $M^n$ is closed and $\Ric_f\geq 0$, then any $\Delta_f$-harmonic $1$-form must be parallel. This is a natural extension of the classical Bochner technique to the weighted setting.

Based on this, we have the following structural result for compact Riemannian manifolds with non-negative Bakry-\'Emery Ricci curvature.
\begin{theorem}[\cite{Lott03}, Theorem 6.6 in \cite{WW09}]\label{thm--structure of compact manifold with ricf geq 0}
    Let $(M,g,f)$ be a compact $n$-dimensional Riemannian manifold with $\Ric_f\geq 0$. Then the following holds:
    \begin{itemize}
        \item Any $\Delta_f$-harmonic $1$-form $\eta$ must be parallel and satisfies $\Ric_f(\eta^{\#},\eta^{\#})=0$ on $M$. In particular, $b_1(M)\leq n$;
        \item There exists a finite cover $\widehat M\to M$, such that we have Riemannian splitting $\widehat M\cong N\times\mathbb T^k$, where $N$ is compact and simply-connected, $T^k$ is a flat torus. Moreover, $\hat f$ is constant along $T^k$-factor and $k\geq b_1(M)$.
    \end{itemize}
\end{theorem}

\subsection{Hopf surfaces}\label{subsec--Hopf surfaces}

Here we review basic theory of Hopf surfaces.

\begin{definition}
    A Hopf surface is a compact connected complex surface, whose universal cover is biholomorphic to $\mathbb C^2\setminus\{(0,0)\}$.
\end{definition}

The following classification result was proved by Kodaira \cite[Theorem 30]{KodairaII}.
\begin{theorem}\label{thm--Kodaira's classification on Hopf surfaces}
    Let $M$ be a Hopf surface, then $\pi_1(M)$ contains a cyclic subgroup $\langle \gamma\rangle$ with finite index. The action of $\gamma$ on $\mathbb C^2\setminus\{(0,0)\}$ under some holomorphic coordinate can be written as 
    \begin{equation}
        (z_1,z_2)\mapsto (\alpha z_1+\lambda z_2^m,\beta z_2)
    \end{equation}
    for $\alpha,\beta,\lambda\in\mathbb C$ and $m\in\mathbb N$ such that $0<\vert\alpha\vert\leq\vert\beta\vert<1$ and $\lambda\cdot(\alpha-\beta^m)=0$.
\end{theorem}

Conversely, there is a topological characterization of Hopf surfaces also due to Kodaira \cite[Theorem 41]{KodairaIII}:
\begin{theorem}
    Let $(M,J)$ be a compact complex surface, such that $b_2(M)=0$ and $\pi_1(M)$ is virtually cyclic, i.e. it contains a finite index subgroup isomorphic to $\mathbb Z$. Then $(M,J)$ is a Hopf surface.
\end{theorem}

Following \cite{KodairaII,HL83}, we use the following standard terminology.

\begin{definition}
A Hopf surface is called \emph{primary} if
\(\pi_1(M)\cong \mathbb Z\), and \emph{secondary} otherwise.
A primary Hopf surface is said to be of \emph{class \(1\)} if
\(\lambda=0\), and of \emph{class \(0\)} otherwise. A \emph{class \(1\)} Hopf surface is called diagonal, if $\vert\alpha\vert=\vert\beta\vert$.
\end{definition}
Note that the distinction between Hopf surfaces of class \(0\) and class \(1\)
is invariant under finite covers. Thus, we may talk about secondary Hopf
surfaces of class \(0\) or class \(1\).

Finally, we mention the explicit computation of automorphism groups of primary Hopf surfaces. The linear case is handled in \cite{namba74} while the non-linear case is completed in \cite{wehler81}. The following table is taken from \cite{wehler81}.
\begin{theorem}[\cite{namba74,wehler81}]
Let $X=(\mathbb C^2\setminus\{(0,0)\})/\langle\gamma\rangle$
be a primary Hopf surface. Then their automorphism groups are summarized
in the following table.

\par\medskip
\begingroup
\centering
\footnotesize
\renewcommand{\arraystretch}{1.6}
\setlength{\tabcolsep}{3pt}

\begin{tabularx}{\linewidth}{
    @{}
    c
    >{\centering\arraybackslash}X
    >{\centering\arraybackslash}X
    @{}
}
\toprule
Type
&
Form of the contraction map
&
$\operatorname{Aut}(X)$
\\
\midrule

$\mathrm{IV}$
&
\(
\begin{gathered}
\gamma(z_1,z_2)
  =(\alpha z_1,\alpha z_2),\\
0<|\alpha|<1
\end{gathered}
\)
&
\(
\mathrm{GL}(2,\mathbb C)/\langle\gamma\rangle
\)
\\
\addlinespace

$\mathrm{III}$
&
\(
\begin{gathered}
\gamma(z_1,z_2)
  =(\beta^m z_1,\beta z_2),\\
m\geq 2,\ 0<|\beta|<1
\end{gathered}
\)
&
\(
\displaystyle
\left\{
\begin{gathered}
(z_1,z_2)\longmapsto
  (a z_1+bz_2^m,dz_2),\\
a,d\in\mathbb C^*,\quad b\in\mathbb C
\end{gathered}
\right\}
\big/
\langle\gamma\rangle
\)
\\
\addlinespace

$\mathrm{II}_a$
&
\(
\begin{gathered}
\gamma(z_1,z_2)
  =(\beta^m z_1+\lambda z_2^m,\beta z_2),\\
m\geq 2,\
\lambda\neq0,\
0<|\beta|<1
\end{gathered}
\)
&
\(
\displaystyle
\left\{
\begin{gathered}
(z_1,z_2)\longmapsto
  (a^m z_1+bz_2^m,az_2),\\
a\in\mathbb C^*,\quad b\in\mathbb C
\end{gathered}
\right\}
\big/
\langle\gamma\rangle
\)
\\
\addlinespace

$\mathrm{II}_b$
&
\(
\begin{gathered}
\gamma(z_1,z_2)
  =(\alpha z_1+\lambda z_2,\alpha z_2),\\
\lambda\neq0,\ 0<|\alpha|<1
\end{gathered}
\)
&
\(
\displaystyle
\left\{
\begin{gathered}
(z_1,z_2)\longmapsto
  (a z_1+bz_2,az_2),\\
a\in\mathbb C^*,\quad b\in\mathbb C
\end{gathered}
\right\}
\big/
\langle\gamma\rangle
\)
\\
\addlinespace

$\mathrm{II}_c$
&
\(
\begin{gathered}
\gamma(z_1,z_2)
  =(\alpha z_1,\beta z_2),\ 0<|\alpha|\leq |\beta|<1,\\
\alpha\neq\beta^m,\  \textit{for every }m\geq2
\end{gathered}
\)
&
\(
\displaystyle
\left\{
\begin{gathered}
(z_1,z_2)\longmapsto(az_1,dz_2),\\
a,d\in\mathbb C^*
\end{gathered}
\right\}
\big/
\langle\gamma\rangle
\)
\\

\bottomrule
\end{tabularx}

\captionof{table}{Automorphism groups of primary Hopf surfaces.}
\label{tab:hopf-automorphism-groups}

\par
\endgroup
\medskip

\end{theorem}

The classical Cartan-Malcev-Iwasawa theorem says that any finite dimensional connected real Lie group $G$ contains a maximal compact subgroup and all maximal compact subgroups are conjugate to each other; cf. \cite[Theorem 1.2 in Chapter VII]{borel98}. Then by standard theory for compact Lie groups \cite{duistermaat-kolk}, we know that all maximal tori of $G$ are conjugate. In particular, they have the same rank, which is usually called the compact rank and is denoted by $\mathrm{rank}_c(G)$. The following lemma should be a standard result. Since we cannot find an explicit reference, we include the computation using Table~\ref{tab:hopf-automorphism-groups}.
\begin{lemma}\label{lem--compact rank of automorphism of Hopf surfaces}
   Let \((M,J)\) be a primary Hopf surface, and let \(G\) denote the identity component of its automorphism group. Then 
    \begin{equation}
        \mathrm{rank}_c(G)=\begin{cases}
3, & \lambda=0,\\
2, & \lambda\neq0.
\end{cases}
    \end{equation}
\end{lemma}
\begin{proof}
We normalize the contraction map \(\gamma\) as above and proceed case by case.

\medskip
\noindent\emph{Class $1$ that is $\lambda=0$.}
There are three sub-cases.

\emph{Type $\mathrm{II}_c$.}
The table gives
\[
G
=
\frac{(\mathbb C^*)^2}
     {\langle(\alpha,\beta)\rangle},\quad 0<|\alpha|<|\beta|<1,\ \alpha\neq\beta^m
\ \text{for every }m\geq2.
\]
Choose \textit{some} logarithm 
\begin{equation}
    \ell_\alpha=\log\alpha,
\ \ell_\beta=\log\beta,
\end{equation}
then the universal covering homomorphism
\begin{equation}
    \mathbb C^2\longrightarrow G,
\qquad
(u,v)\longmapsto
\bigl[(e^u,e^v)\bigr]
\end{equation}
has kernel
\begin{equation}
   \Lambda_{\alpha,\beta}
=
\left\langle
(2\pi i,0),\,
(0,2\pi i),\,
(\ell_\alpha,\ell_\beta)
\right\rangle_{\mathbb Z}. 
\end{equation}
It is not hard to see these three vectors are $\mathbb R$-linearly independent. Consequently,
\begin{equation}
    G
\cong
\mathbb C^2/\Lambda_{\alpha,\beta}
\cong
\mathbb T^3\times\mathbb R
\end{equation}
as a real Lie group. Therefore $\mathrm{rank}_c(G)=3$.

\emph{Type $\mathrm{III}$.}
Suppose
\begin{equation}
    \gamma(z_1,z_2)
=
(\beta^m z_1,\beta z_2),
\ m\geq2,\ 0<|\beta|<1.
\end{equation}
Denote by
\begin{equation}
    H=\left\{
F_{a,b,d}(z_1,z_2)
=
(az_1+bz_2^m,dz_2):
a,d\in\mathbb C^*,\ b\in\mathbb C
\right\}.
\end{equation}
Then the table gives
\begin{equation}
    G=H/\langle\gamma\rangle.
\end{equation}

Consider the homomorphism
\begin{equation}
    \pi:H\longrightarrow(\mathbb C^*)^2,
\qquad
\pi(F_{a,b,d})=(a,d),
\end{equation}
whose kernel is
\begin{equation}
    U
=
\left\{
(z_1,z_2)\longmapsto
(z_1+bz_2^m,z_2):
b\in\mathbb C
\right\}
\cong(\mathbb C,+).
\end{equation}
Since $\pi(\gamma)=(\beta^m,\beta)$, the homomorphism $\pi$ then induces an exact sequence
\begin{equation}\label{equa--exact sequence of type III}
    1\longrightarrow U
\longrightarrow G
\longrightarrow
\frac{(\mathbb C^*)^2}
     {\langle(\beta^m,\beta)\rangle}
\longrightarrow1.
\end{equation}

The diagonal subgroup of $H$, obtained by setting $b=0$, descends to a subgroup of $G$ isomorphic to
   $ \frac{(\mathbb C^*)^2}
     {\langle(\beta^m,\beta)\rangle}.$
By the calculation in the type $\mathrm{II}_c$ case, this subgroup
contains a compact torus of dimension $3$. Hence $\mathrm{rank}_c(G)\geq 3$.
For the opposite direction, let $T\subset G$ be any compact torus.
The additive group $U\cong(\mathbb C,+)$ has no nontrivial compact
subgroup, so $T\cap U=\{\mathrm{Id}\}$. It follows that the projection of $T$ to $(\mathbb C^*)^2/
\langle(\beta^m,\beta)\rangle$ is injective. Since the latter group has compact rank $3$, we see that $\dim T\leq 3$. Thus $\mathrm{rank}_c(G)=3$.

\emph{Type $\mathrm{IV}$.}
Suppose
\begin{equation}
    \gamma(z_1,z_2)=(\alpha z_1,\alpha z_2),\ 0<|\alpha|<1,
\end{equation}
then
\begin{equation}
    G
=
\mathrm{GL}(2,\mathbb C)/\langle\alpha I\rangle.
\end{equation}
The diagonal subgroup gives an inclusion
\begin{equation}
    \frac{(\mathbb C^*)^2}
     {\langle(\alpha,\alpha)\rangle}
\subset G,
\end{equation}
which contains a $3$-torus using the same argument before. Hence $\mathrm{rank}_c(G)\geq 3$.
To obtain the upper bound, consider the exact sequence
\begin{equation}\label{eq-exact sequence for groups}
    1\longrightarrow
\frac{\mathbb C^*}{\langle\alpha\rangle}
\longrightarrow
\frac{\mathrm{GL}(2,\mathbb C)}
     {\langle\alpha I\rangle}
\longrightarrow
\mathrm{PGL}(2,\mathbb C)
\longrightarrow1,
\end{equation}
whose kernel is
\begin{equation}
    \frac{\mathbb C^*}{\langle\alpha\rangle}
\cong
\frac{\mathbb C}
{\langle2\pi i,\ell_\alpha\rangle_{\mathbb Z}}
\cong\mathbb T^2.
\end{equation}
Moreover, a maximal compact subgroup of
$\mathrm{PGL}(2,\mathbb C)$ is $\mathrm{PU}(2)\cong\mathrm{SO}(3)$, whose maximal tori are one-dimensional. Then using the exact sequence \eqref{eq-exact sequence for groups}, we obtain that $\mathrm{rank}_c(G)$ is at most 3.
Therefore $\mathrm{rank}_c(G)=3$.

\medskip
\noindent
\emph{Class $0$ that is $\lambda\neq 0$.}
Types $\mathrm{II}_a$ and $\mathrm{II}_b$ will be treated
simultaneously. Their contraction map has the form
\begin{equation}
    \gamma(z_1,z_2)
=(\beta^m z_1+\lambda z_2^m,\beta z_2),
\ \lambda\neq0,\ m\geq1.
\end{equation}
Here type $\mathrm{II}_a$ corresponds to $m\geq2$, while type
$\mathrm{II}_b$ corresponds to $m=1$.

By the table,
\begin{equation}
    G=H/\langle\gamma\rangle,
\end{equation}
where
\begin{equation}
    H=\left\{
F_{a,b}(z_1,z_2)
=
(a^m z_1+bz_2^m,az_2):
a\in\mathbb C^*,\ b\in\mathbb C
\right\}.
\end{equation}

Let $s=\frac{b}{a^m}$. We write $F_{a,b}$ as
\begin{equation}
    \Phi(a,s)(z_1,z_2)
=\bigl(a^m(z_1+s z_2^m),az_2\bigr).
\end{equation}
A direct calculation gives
\begin{equation}
    \Phi(a,s)\circ\Phi(c,t)=\Phi(ac,s+t),
\end{equation}
thus $H\cong\mathbb C^*\times\mathbb C$, where the first factor is the multiplicative group and the
second factor is the additive group.

Under this identification, the deck generator corresponds to
\begin{equation}
   \gamma
\longleftrightarrow
\left(
\beta,\frac{\lambda}{\beta^m}
\right),
\end{equation}
thus
\begin{equation}
    G
\cong
\frac{\mathbb C^*\times\mathbb C}
{\left\langle
\left(\beta,\lambda/\beta^m\right)
\right\rangle}.
\end{equation}
Choose a logarithm $\ell_\beta=\log\beta$. Similarly to previous argument, we consider the universal covering homomorphism
\begin{equation}
    \mathbb C^2\longrightarrow G,
\qquad
(u,v)\longmapsto[(e^u,v)]
\end{equation}
whose kernel is
\begin{equation}
    \Lambda_0
=
\left\langle
(2\pi i,0),\,
\left(
\ell_\beta,\frac{\lambda}{\beta^m}
\right)
\right\rangle_{\mathbb Z}.
\end{equation}
One can check that the two generators are $\mathbb R$-linearly independent. Therefore $\Lambda_0$
has real rank $2$, and
\begin{equation}
    G
\cong
\mathbb C^2/\Lambda_0
\cong
\mathbb T^2\times\mathbb R^2
\end{equation}
as a real Lie group. It follows that $\mathrm{rank}_c(G)=2$.
\end{proof}

\section{Proof of main theorems}

\subsection{Existence of three commuting Killing fields}
We observe first that any unknown non-static steady pluriclosed soliton metric admits three linearly independent commuting real holomorphic Killing fields.

We start with the following tensor algebra lemma:
\begin{lemma}[Four-dimensional three-form algebra]\label{lem--tensor algebra lemma}
\label{lem:four-dimensional-algebra}
Let $(V,g)$ be an oriented real four dimensional vector space endowed with a positive definite metric $g$. If $\theta$ is a one-form and
$H=-*\theta$, then for every vector $X\in V$,
\begin{equation}\label{eq:four-dimensional-algebra}
 \iota_XH=-*(\theta\wedge X^\flat),
 \qquad
 H^2=2\bigl(|\theta|^2g-\theta\otimes\theta\bigr).
\end{equation}
Thus $H^2\geq0$ and, wherever $H\ne0$,
\begin{equation}\label{eq-zero case}
 \ker H^2=\{X:\iota_XH=0\}=\mathbb R\theta^\sharp.
\end{equation}
\end{lemma}
\begin{proof}
    In dimension four, the  Hodge star operator satisfies
\begin{equation}
     *^2=(-1)^{k(4-k)}=-1\quad(k=1,3).
\end{equation}

Without loss of generality, suppose $\theta\ne0$ and
choose an oriented orthonormal coframe $(e^1,e^2,e^3,e^4)$ with
$\theta=|\theta|e^1$.  Writing $X=\sum_i x_i e_i$ gives
\[
 \begin{split}
 \iota_XH
 &=-|\theta|(x_2e^{34}-x_3e^{24}+x_4e^{23})\\
 &=-*\bigl(|\theta|(x_2e^{12}+x_3e^{13}+x_4e^{14})\bigr)
 =-*(\theta\wedge X^\flat).
 \end{split}
\]
 Since $*$ is an isometry,
\[
 |\iota_XH|^2=|\theta\wedge X^\flat|^2
 =|\theta|^2|X|^2-\theta(X)^2.
\]
Using the definition of $H^2$, we can compute that 
\begin{equation}
    H^2(X,X)=2\vert\iota_XH\vert^2,
\end{equation}
thus concluding the identity \eqref{eq:four-dimensional-algebra} and \eqref{eq-zero case}.
\end{proof}

In particular, the semi-positive definiteness of $H^2$ implies $\Ric_f\geq0$ for a steady pluriclosed soliton. Moreover, by soliton equation we must have
\begin{equation}\label{equa--kernel of ricci is equal to kernel of H^2}
    \ker(\Ric_{f,p})=\ker (H^2_p)
\end{equation}
for every $p\in M$. 
As a result, we obtain:
\begin{corollary}\label{cor--rough complex structure determination}
    Let $(M,J,g,f)$ be a compact steady soliton surface. Then either $H\equiv 0$, where $(M,J,g)$ is a K\"ahler Ricci-flat surface and $f$ is a constant, or $H\not\equiv 0$, where $(M,J)$ is a  Hopf surface.
\end{corollary}
\begin{proof}
    If $H\equiv 0$, then $g$ is K\"ahler, and the soliton equation becomes 
    \begin{equation}
        \Ric+\nabla^2 f=0.
    \end{equation}
    It's well-known that a compact steady gradient Ricci soliton is Ricci-flat. In fact, to see this we use
    \begin{equation}
        \Delta_f R=2\vert\Ric\vert^2,
    \end{equation}
    so integration against weighted volume $e^{-f}d\mu_g$ gives
    \begin{equation}
        \int\vert\Ric\vert^2 e^{-f}d\mu_g=0.
    \end{equation}

    Suppose otherwise $H\neq 0$ at $p\in M$. By Theorem \ref{thm--structure of compact manifold with ricf geq 0} and Lemma \ref{lem--tensor algebra lemma} applied at $p$, we must have $b_1(M)\leq 1$. If $b_1(M)=0$, then $(M,J)$ is K\"ahler \cite[Chapter IV, Theorem 3.1]{barth2003}, so by \cite[Proposition 3.5]{Streets19} the soliton metric must be a K\"ahler Ricci-flat metric with $f\equiv\mathrm{const}$, contradicting to our assumption $H\neq 0$. Thus $b_1(M)=1$, and by Theorem \ref{thm--structure of compact manifold with ricf geq 0}, there exists a non-zero parallel Killing $1$-form $\eta$ on $M$. The existence of nowhere vanishing vector field on $M$ implies that $\chi(M)=0$ by Poincar\'e--Hopf theorem. Therefore by Poincar\'e duality, we see $b_2(M)=0$.

    Finally we determine the complex structure of $(M,J)$. By \cite[Theorem 6.6]{WW09}, there exists a finite cover $\widehat M\to M$, such that $\widehat M$ admits a Riemannian splitting $N\times T^k$, where $N$ is compact and simply-connected and $T^k$ is the flat torus. Moreover, $\hat f$ is constant along flat factors. Since $b_1(M)=1$, we must have $k\geq 1$.
    
    We claim that $k=1$. In fact, let $E$ be any parallel vector field tangential to $T^k$-factor. Since the metric splits and $\hat f$ is constant along $T^k$-factor, by soliton equation we have $\iota_E\tilde H=0$ on $\tilde M$. Lemma \ref{lem--tensor algebra lemma} then implies that $k=1$. Thus
    \begin{equation}
        \pi_1(\widehat M)\cong\pi_1(S^1)=\mathbb Z,
    \end{equation}
    since $N$ is simply-connected. The covering map $\widehat M\to M$ is finite, so $\pi_1(M)$ is virtually cyclic, i.e. it contains a finite index subgroup isometric to $\mathbb Z$. Since $b_2(M)=0$, by Kodaira's characterization \cite[Theorem 41]{KodairaIII}, $(M,J)$ must be a Hopf surface.
\end{proof}

Let $K=\eta^{\#}$ be the nowhere vanishing parallel vector field constructed via harmonic $1$-form in Corollary \ref{cor--rough complex structure determination}. Recall that it satisfies $\nabla^g K=0$ and $\iota_K H=0$. Then we have:
\begin{proposition}\label{prop--real holomorphic Killing fields from harmonic 1-form}
    Both $K$ and $JK$ are real holomorphic Killing fields and commute with each other.
\end{proposition}\label{prop--holomorhic killing vector fields}
\begin{proof}
    Let $\nabla^B$ be the Bismut connection, defined as
    \begin{equation}
       \nabla^B=\nabla^g-\frac12 g^{-1}d^c\omega.
        \end{equation}
        Equivalently, 
        \begin{equation}\label{eq--relation between LC and Bismut}
            \nabla^B_YZ=\nabla^g_YZ+\frac12(H(Y,Z,\cdot))^{\#}.
        \end{equation}
         Because $K$ is parallel, we have
        \begin{equation}
            (\mathcal L_KJ)Y=\nabla^g_K(JY)-J\nabla_KY.
        \end{equation}
        Using $\nabla^BJ=0$ and $\iota_KH=0$, we have
        \begin{equation}\label{eq-contraction zero imply}
            \nabla^g_K(JY)-J\nabla_KY=-\frac12(H(K,JY,\cdot))^{\#}+\frac12J(H(K,Y,\cdot)^{\#})=0.
        \end{equation}
        Therefore $\mathcal L_KJ=0$, i.e. $K$ is holomorphic. It follows that $JK$ is also holomorphic.

        Next we show $JK$ is also Killing. Since $K$ is parallel and $\nabla^B J=0$, by \eqref{eq--relation between LC and Bismut}, we have
        \[
            \nabla^g_Y(JK)=-\frac12(H(Y,JK,\cdot))^{\#},\]
        from which we obtain
        \begin{equation}
            g(\nabla^g_Y(JK),Z)=-\frac12 H(Y,JK,Z).
        \end{equation}
        Since $\nabla^g$ is metric-compatible and torsion-free, we get
        \begin{equation}
            \begin{aligned}
                (\mathcal L_{JK}g)(Y,Z)
&=g(\nabla_Y^g(JK),Z)+g(Y,\nabla^g_Z(JK))\\
                &=-\frac12 H(Y,JK,Z)-\frac12 H(Z,JK,Y).
            \end{aligned}
        \end{equation}
        Since $H$ is a $3$-form, we conclude that $\mathcal L_{JK}g=0$.

        Finally use $\mathcal L_KJ=0$, we conclude that
       \(
            [K,JK]=(\mathcal L_KJ)(K)+J([K,K])=0.\)
\end{proof}

It is also known that the soliton metric itself produces a natural symmetry:
\begin{proposition}[Proposition 4.1 in \cite{Streets-Ustinovskiy22}]\label{prop--Lie derivative of soliton vector field}
    Consider the vector field
    \begin{equation}\label{eq-definition of V}
        V\coloneqq\frac12(\theta^\#-\nabla f).
    \end{equation}
    Then it satisfies
    \begin{equation}\label{eq--symmetries}
            \mathcal L_VJ=0,
            \quad \mathcal L_{JV}g=0,\quad
            \mathcal L_V\omega=\rho_B^{1,1}.
    \end{equation}
    In particular, $JV$ is holomorphic and Killing.
\end{proposition}

Now we study the relation between these holomorphic Killing fields.
\begin{proposition}\label{prop--linearly dependence of Killing fields}
    All four vector fields $K,JK,V,JV$ commute with each other. If they are linearly dependent, then the metric is static, i.e. $\rho_B^{1,1}\equiv 0$.
\end{proposition}
\begin{proof}
    Let $A$ be either $K$ or $JK$, which is holomorphic and Killing. Thus its flow preserves $g, J, \omega$, and hence $H$ and $\theta$. It follows that $[A,\theta^\#]=\mathcal L_A\theta^\#=0$.

    Next we show that $A$ preserves the soliton potential $f$. We take Lie derivative of the equation
    \begin{equation}
        \Ric+\nabla^2 f-\frac14 H^2=0,
    \end{equation}
    and since $A$ is Killing, we get
    \begin{equation}
        0=\mathcal L_A(\nabla^2 f)=\nabla^2(Af).
    \end{equation}
    Thus $Af$ is a constant. Note that $
        \mathrm{div}(A)=\frac12\mathrm{tr}_g(\mathcal L_A g)=0,$
    so by divergence theorem
    \begin{equation}
        \int_M Afd\mu_g=\int_M\mathrm{div}(fA)d\mu_g-\int_M f\mathrm{div}(A)d\mu_g=0.
    \end{equation}
    Therefore $Af\equiv 0$. Using the Killing condition again we see $$[A,\nabla f]=\mathcal L_X(g^{-1}df)=0,$$ so we conclude that $[A,V]=0$. Then using $\mathcal{L}_AJ=0$, we obtain
       $[A,JV]=0.$

    Finally we study their linear dependence. Consider the holomorphic bivector
    \begin{equation}
        \sigma=K^{1,0}\wedge V^{1,0}\in H^0(M, K_M^{-1}),
    \end{equation}
    whose zero locus is either a divisor, or the whole space. In the first case, these vectors are generically linearly independent. Suppose not, i.e. $\sigma\equiv 0$. Since $K^{1,0}$ is nowhere vanishing, we must have $V^{1,0}=h K^{1,0}$ for a global holomorphic function $h$. Since $M$ is compact, $h\equiv a+\sqrt{-1}b$ is a constant. It follows that
    \begin{equation}
        V=aK+bJK
    \end{equation}
     is a Killing field. By Proposition \ref{prop--Lie derivative of soliton vector field}, we conclude that $\rho_B^{1,1}\equiv 0$, i.e. the metric is static.
\end{proof}
\begin{proposition}\label{prop:staticity}
The soliton metric is static if and only if $\nabla f\equiv0$. If the metric is static, then $\theta$ is a parallel $1$-form and $\Ric(\theta^\#,\theta^\#)=0$. In particular, $\vert\theta\vert$ is constant.
\end{proposition}

\begin{proof}
Recall that we assume $g$ is pluriclosed, i.e. $dH=0$. It follows that
\begin{equation}
    \delta\theta=-*d*\theta=*dH=0.
\end{equation}
On the other hand, the soliton equation says that
\begin{equation}
 \delta H+\iota_{\nabla f}H=0.
\end{equation}
 By Lemma \ref{lem:four-dimensional-algebra} we have $\iota_XH=-*(\theta\wedge X^\flat)$ and therefore we get
     \(\delta H=*(\theta\wedge df),\)
 thus
\begin{equation}\label{equa--dtheta=df wedge dtheta}
      d\theta=df\wedge\theta.
\end{equation}

If the soliton is static, then \eqref{prop--Lie derivative of soliton vector field} shows that $V$
preserves both $J$ and $\omega$, and hence $g$.  Thus
$\operatorname{div}V=0$. On the other hand, \eqref{eq-definition of V} and
$\delta\theta=0$ give
\[
 0=2\operatorname{div}V
   =\operatorname{div}\theta^\#-\operatorname{div}(\nabla f)
   =-\operatorname{div}(\nabla f).
\]
Thus $f$ is harmonic, and compactness of $M$ gives
$\nabla f\equiv0$.   

Conversely, suppose that $\nabla f\equiv0$. Then Lemma \ref{lem--tensor algebra lemma} implies that $\Ric_g\geq 0$. The equation \eqref{equa--dtheta=df wedge dtheta} implies $d\theta=df\wedge\theta=0$, while pluriclosedness gives
$\delta\theta=0$. Therefore $\theta$ is harmonic, so the standard Bochner technique yields
$$\nabla\theta=0\quad  \Ric(\theta^\#, \theta^\#)=0.$$ In particular, $|\theta|$ is constant, and \eqref{eq-definition of V} gives
$V=\frac 12 \theta^\#$ is parallel and hence Killing. Then combing with \eqref{eq--symmetries},  we obtain  $\rho_B^{1,1}=0$.
\end{proof}

\subsection{Proof of Theorem \ref{thm--classification of complex structure}}\label{sec--proof of complex structure classification}

In this subsection, we finish the classification of complex structures. By passing to a finite cover and pulling back the soliton metric, we assume $(M,J)$ be a \textit{primary} Hopf surface admitting a steady pluriclosed soliton.

If $g$ is static, i.e. $\rho_B^{1,1}\equiv 0$, then the classification result in \cite[Theorem 2]{GI97} and \cite[Theorem 1.4]{streets-tian12} implies that it must be the standard Hopf metric, so in particular the classification of complex structure is done. We include a self-contained proof here for readers' convenience. Our goal is to show that $(M,J)$ is a Hopf surface, whose any finite primary cover is of class~$1$ and diagonal.

By Proposition \ref{prop:staticity}, we know that $f$ is constant and $\theta$ is a non-zero parallel $1$-form. We pass to the universal cover $\widetilde M$, and pull back the soliton metric $g$, denoted by $\widetilde g$. Then there is a de Rham splitting
\begin{equation}
    (\widetilde M,\widetilde g)\cong(N\times\mathbb R, \widetilde g=g_N\oplus dt^2),
\end{equation}
where $\theta=\vert\theta\vert dt$ and $\Ric_g(\theta^{\#},\theta^{\#})=0$. Using Lemma \ref{lem--tensor algebra lemma} it is then not hard to see 
\begin{equation}
    \Ric_{g_N}=\frac{\vert\theta\vert^2}{2}g_N.
\end{equation}
Since $N$ is a simply-connected $3$-fold, it must be the standard round metric with constant sectional curvature.
The complex structure can be identified in the following way. Note that $g$ is a  Vaisman metric since $\theta$ is parallel. The standard theory then shows that $\widetilde M$ admits a K\"ahler cone metric with link diffeomorphic to $N$; cf. \cite{Vaisman79}, see also \cite[Section 9]{Belgun00}.

More precisely, we consider the conformal change
\begin{equation}\label{eq:flat-cone}
 r=\frac2{c} e^{-ct/2},
 \
 g_{\mathrm{cone}}=e^{-ct}\widetilde g
 =dr^2+r^2\frac{c^2}{4}g_N,
\end{equation}
which is K\"ahler since $d(e^{-ct}\widetilde \omega)=0$. Here we denote by $c=\vert\theta\vert_g>0$. The link is the unit round sphere, so it is not hard to see $\widetilde M$ is biholomorphic to $\mathbb C^2\setminus\{(0,0)\}$ using the completion of flat cone metric. This proves that $(M,J)$ is a Hopf surface.

Next we prove that any finite primary cover must be class~$1$ and diagonal. The proof uses the static metric. Denote by $W=\mathbb C^2\setminus\{(0,0)\}$. We fix a biholomorphic identification of $\widetilde M$ with $W$, so the description of K\"ahler cone metric can be rephrased as follows: for any static metric $g$, there exists a biholomorphism $A:W\to W$ and $\lambda>0$, such that
\begin{equation}
    \widetilde g=\lambda A^*g_H,\ g_H=\frac{g_{Euc}}{\vert z\vert^2}.
\end{equation}

Let $\widehat M=W/\langle\gamma\rangle$ be a finite primary cover of $M$. We are going to prove that there exists a suitable coordinate change of $W$, such that $\gamma$ is diagonal.
Since $\gamma$ is a deck transform, $\gamma^*\widetilde g=\widetilde g$. Therefore $A\gamma A^{-1}$ is a holomorphic isometry of standard Hopf metric $g_H$. It follows that we can choose $\rho\in(0,1)$ and $U\in U(2)$, such that $A\gamma A^{-1}=\rho U$. We diagonalize $U$, so we can find $V\in U(2)$, such that $VUV^{-1}$ is a diagonal unitary matrix $\mathrm{diag}(e^{i\theta_1},e^{i\theta_2})$. Thus we write $\gamma$ under the new coordinate $C=VA$, and consider the new gauge $B=AC^{-1}$. Under such choice we have
\begin{equation}
    \gamma^{\prime}=C\gamma C^{-1}=\rho\cdot \mathrm{diag}(e^{i\theta_1},e^{i\theta_2}),\ \tilde g^{\prime}=\lambda(C^{-1})^* B^* g_H.
\end{equation}
Therefore by rebelling, we can simply assume $\gamma$ is diagonal and $\widetilde g=\lambda A^* g_H$. It also shows that $\vert\alpha\vert=\vert\beta\vert=\rho$. This proves that $\widehat M$ is a diagonal Hopf surface.

Then we consider the case that $g$ is non-static. Therefore the four vectors in Proposition \ref{prop--linearly dependence of Killing fields}, are linearly independent. Denote by
\begin{equation}
    G=\mathrm{Aut}^0(M,J),\ G_g=G\cap\mathrm{Isom}(M,g).
\end{equation}
Then we consider the Abelian subalgebra:
\begin{equation}
    \mathfrak a=\mathrm{Span}_\mathbb R\{K,JK,JV\}\subset\mathfrak{aut}(M,J),
\end{equation}
whose flow lines lie in the compact group $G_g$, thus generating a compact torus $T_g$ of dimension at least $3$.
The proof of Theorem \ref{thm--classification of complex structure} now follows from Lemma \ref{lem--compact rank of automorphism of Hopf surfaces}, since class~$0$ primary Hopf surface has compact rank less than $3$. In particular, we also conclude that $T_g$ is a maximal torus when $(M,J)$ is primary and $g$ is non-static.



\subsection{Proof of Theorem \ref{thm--uniqueness of soliton metric}}\label{subsec--proof of uniqueness of soliton metric}

In this section, we prove that all steady pluriclosed solitons $(M,J,g)$ on class $1$ Hopf surfaces must be the one constructed by Streets \cite{Streets19}. This is divided into three cases.

\noindent\textbf{Case I.} $(M,J)$ is primary and $g$ is non-static. In this case, we will show uniqueness and also the surface is \textit{non-diagonal}, that is  $|\alpha|\neq |\beta|$.
We follow the toric method in \cite{Streets-Ustinovskiy21}, so the key is to transfer from the above coordinate-free construction to explicit coordinate-dependent one.
The following result follows from our discussion in Section \ref{subsec--Hopf surfaces}.

\begin{proposition}[\cite{namba74}, see also Section 2.1 in \cite{Streets-Ustinovskiy21}]
    Let $(M,J)$ be a primary Hopf surface of class $1$, and we work over the open subset $\{z_1z_2\neq 0\}$. Consider the logarithmic coordinates
\begin{equation}
    w_j=x_j+\sqrt{-1}y_j=\log z_j.
\end{equation}
Define
\begin{equation}
    R_j=\partial_{x_j}, T_j= J R_j=\partial_{y_j}.
\end{equation}
Write
\begin{equation}
    a=\Re\log\alpha,\ b=\Re\log\beta,
\end{equation}
and consider
\begin{equation}
    Z=aR_1+b R_2.
\end{equation}
Then the maximal compact torus $T_{\mathrm {std}}\subset G$ has Lie algebra
\begin{equation}
    \mathfrak t_{\mathrm {std}}=\mathrm{span}_{\mathbb R}\{T_1, T_2, Z\}.
\end{equation}
\end{proposition}

By Cartan-Malcev-Iwasawa theorem \cite[Theorem 1.2 in Chapter VII]{borel98}, there exists $\Phi\in G$, such that $\Phi^{-1} T_g\Phi=T_{\mathrm {std}}$. Therefore after replacing $g$ by $\Phi^*g$, we may assume $\mathfrak a=\mathfrak t_{\mathrm {std}}$.

We use the following elementary observation:
\begin{lemma}
    The only $J$-invariant real $2$-planes in $\mathfrak t_{\mathrm {std}}$ is $\mathrm{span}_{\mathbb R}\{Z,JZ\}$.
\end{lemma}
\begin{proof}
    Let $\Sigma\subset\mathfrak t_{\mathrm {std}}$ be a real $J$-invariant $2$-plane. Take $0\neq X\in\Sigma$, and write
    \begin{equation}
        X=pT_1+qT_2+rZ.
    \end{equation}
    Then
    \begin{equation}
        JX=-pR_1-qR_2+r(aT_1+bT_2)\in\Sigma\subset\mathfrak t_{\mathrm {std}}.
    \end{equation}
    Therefore $-pR_1-qR_2$ must be proportional to $Z$, so we can find $s\in\mathbb R$ such that
    \begin{equation}
        p=-sa, q=-sb.
    \end{equation}
    Thus
    \begin{equation}
        JX=-s(aR_1+bR_2)+r JZ=-s Z+rJZ.
    \end{equation}
    The result then follows.
\end{proof}

Since $\mathrm{span}_\mathbb R\{K,JK\}$ is also $J$-invariant, we must have 
\begin{equation}\label{equa--span of K JK}
    \mathrm{span}_\mathbb R\{K,JK\}=\mathrm{span}_{\mathbb R}\{Z,JZ\}.
\end{equation}
In particular, $Z, JZ$ are Killing fields of $g$. 
Following \cite{Streets-Ustinovskiy21}, we consider the linear coordinate change
\begin{equation}
    u_1=\frac{b}{a}w_1-w_2,\  u_2=w_2.
\end{equation}
Set $Y\coloneqq\partial_{\Re u_1}$, then 
\begin{equation}
    \begin{aligned}
        &Y=\frac{a}{b}R_1,\ JY=\frac{a}{b}T_1,\\
        & Z=b\partial_{\Re u_2},\  JZ=b\partial_{\Im u_2}.
    \end{aligned}
    \end{equation}
    It follows that
    \begin{equation}
        \mathfrak t_{\mathrm {std}}=\mathrm{span}_\mathbb R\{JY,Z,JZ\}.
    \end{equation}
    Since $JV\in\mathfrak t_{\mathrm {std}}$, we can write
    \begin{equation}
        JV=\mu JY+c Z+d JZ.
    \end{equation}
    By \eqref{equa--span of K JK} and linear independence of $K, JK, JV$, we must have $\mu\neq 0$. Note that $JV$ is also Killing, so
    \begin{equation}
        0=\mathcal L_{JV}g=\mu\mathcal L_{JY}g+c\mathcal L_Zg+d\mathcal L_{JZ}g=\mu\mathcal L_{JY}g.
    \end{equation}
    Therefore $JY$ is also Killing. Hence $g$ is invariant under
    \begin{equation}
        \partial_{\Im u_1},\ \partial_{\Im u_2},\ \partial_{\Re u_2},
    \end{equation}
    which is exactly the symmetry used in \cite{Streets-Ustinovskiy21}. Therefore the soliton equation reduces to the ODE in \cite{Streets-Ustinovskiy21}.

    More precisely, the soliton metric is given by \cite[Proposition 2.5]{Streets-Ustinovskiy21}
    \begin{equation}
        \omega=\ii\left(\frac{b^2}{a^2}\frac{k}{\vert z_1\vert^2}dz_1\wedge d\bar z_1+\frac{1-k}{\vert z_2\vert^2}dz_2\wedge d\bar z_2\right),
    \end{equation}
    where $x=\frac{b}{a}\log\vert z_1\vert^2-\log\vert z_2\vert^2$ and $k=k(x)$ is a solution of the ODE
    \begin{equation}\label{eq--ode}
        k^{\prime}=k(1-k)\big((1-\frac{a}{b})k+\frac{a}{b}\big).
    \end{equation}

    To derive this ODE, there are two places we need to normalize constants. First, we need to normalize the constant $p$ to $1$ as in \cite[Section 2.2, top of Page 1899]{Streets-Ustinovskiy21}. For this we multiply the metric $g$ by a positive constant $\lambda$. Then the soliton equation reduces to the ODE of one variable function $k(x)$ in \cite[Proposition 2.5]{Streets-Ustinovskiy21}. The ambiguity then comes from replacing $k(x)$ by $k(x+\mathrm{const})$, or equivalently, choosing different constants in \cite[equation (2.17)]{Streets-Ustinovskiy21}. Recall that 
    \begin{equation}
        x=u_1+\bar u_1,
    \end{equation}
    so $x$-translation corresponds to flow line of $Y$. The corresponding global biholomorphism is 
    \begin{equation}
        \varphi_s(u_1,u_2)=(u_1+s,u_2),
    \end{equation}
    which in original coordinate is given by
    \begin{equation}\label{eq-extra biholomorphism}
        \varphi_s(z_1,z_2)=(e^{\frac{a}{b}s}z_1,z_2).
    \end{equation}
    It commutes with deck transformation
      \(  \gamma(z_1,z_2)=(\alpha z_1,\beta z_2),\)
    so descends to a biholomorphism of $(M,J)$. 
    Therefore, up to another biholomorphism $\varphi_s$, the metric is equal to $g_0$ constructed in \cite{Streets19,Streets-Ustinovskiy21}. The uniqueness of soliton potential then follows from soliton equation.

Note that from the construction, we have
\[
 a=\log|\alpha|,
 \qquad b=\log|\beta|.\]
From \eqref{eq--ode}, we see that if $a=b$, then the metric corresponding to the solution takes the form (cf. \cite[Bottom of Page 1902]{Streets-Ustinovskiy21})
\[
    \frac{|dz_1|^2+C|dz_2|^2}
    {|z_1|^2+C|z_2|^2},\]
which is static. Since $g$ is assumed to be non-static in this case, it follows that
\[
    |\alpha|\neq|\beta|.
\]

\ 

\noindent\textbf{Case II.} $(M,J)$ is secondary and $g$ is non-static. Let $(g_i,f_i)$, $i=1,2$, be two non-static solitons.  Taking a finite regular primary cover
\[
 \pi\colon M_\gamma\longrightarrow M,
 \qquad F=\operatorname{Deck}(\pi),
\]
where $\gamma=\operatorname{diag}(\alpha,\beta)$.  
Since we can lift holomorphic vector fields from $M$ to $M_\gamma$, 
Lemma \ref{lem--compact rank of automorphism of Hopf surfaces} bounds the compact rank of 
$\Aut^0(M)$ by three. Hence the canonical tori $T_{g_i}$ supplied by
Proposition~\ref{prop--linearly dependence of Killing fields} are maximal three-tori and therefore conjugate
in $\Aut^0(M)$. After conjugating by an element in $\Aut^0(M)$, we may assume $T_{g_1}=T_{g_2}$.

 The holomorphic Killing vector fields are natural under
pullback, so for pullback metrics, we have
\(
 T_{\pi^*g_1}=T_{\pi^*g_2}
\) is a maximal torus inside $\Aut^0(M_\gamma)$. Since $\pi^*g_i$ are $F$-invariant, we know that this torus centralizes $F$. Therefore after conjugating by an element in $\Aut^0(M_\gamma)$, we may assume  
both metrics are $T_{\mathrm{std}}$- and $F$-invariant, and $F$ centralizes
$T_{\mathrm{std}}$.

Then we need the following elementary identity for class~$1$ primary Hopf surfaces:
\begin{equation}\label{eq:standard-torus-centralizer}
 C_{\Aut(M_\gamma)}(T_{\mathrm{std}})
 = (\C^*)^2/\langle\gamma\rangle.
\end{equation}
To see this, for $\vartheta=(\vartheta_1,\vartheta_2)\in
(\mathbb R/2\pi\mathbb Z)^2$, let
\(
 R_\vartheta(z_1,z_2)
 =\bigl(e^{\sqrt{-1}\vartheta_1}z_1,
        e^{\sqrt{-1}\vartheta_2}z_2\bigr).
\)
Since $R_\vartheta$ commutes with $\gamma$, it descends to 
element in  $T_{\mathrm{std}}$.
If $\varphi$ centralizes $T_{\mathrm{std}}$ and
$\widehat\varphi\in\Aut(\mathbb C^2\setminus\{(0,0)\})$ is a lift, then
\[
 \widehat\varphi R_\vartheta
 =\gamma^{k(\vartheta)}R_\vartheta\widehat\varphi.
\]
The integer $k(\vartheta)$ is locally constant and $k(0)=0$, so connectedness
of $U(1)\times U(1)$ gives $k(\vartheta)=0$.    Hartogs' theorem extends
$\widehat\varphi$ and its inverse to  automorphisms of $\C^2$ fixing
the origin; the commutativity with the $U(1)\times U(1)$-action then gives
\[
 \widehat\varphi(z_1,z_2)=(c_1z_1,c_2z_2),
 \qquad c_1,c_2\in\C^*.
\]
The reverse inclusion in \eqref{eq:standard-torus-centralizer} is immediate.
Thus $F$ is diagonal and the biholomorphism in  \eqref{eq-extra biholomorphism}
 commutes with $F$ and hence descends to $M$. This proves metric uniqueness in Case~II.

\ 

\noindent\textbf{Case III.} $g$ is static.  
Let $\widehat M=W/\langle\gamma\rangle$ be a finite primary cover of $M$. It was proved in Section \ref{sec--proof of complex structure classification} that we can assume
\begin{equation}
    \gamma=\rho\cdot\mathrm{diag}(e^{i\theta_1},e^{i\theta_2})
\end{equation}
for some $\rho\in(0,1)$. Moreover, there exists a biholomorphism $A: W\to W$ and $\lambda>0$, such that
\begin{equation}
    \widetilde g=\lambda A^* g_H, \quad  A\gamma A^{-1}=\rho U,\ U\in U(2).
\end{equation}

The Hartogs extension theorem shows that $A$ and $A^{-1}$ can be extended to 0 with $A(0)=0$. Now we expand this identity using homogeneous expansion of $A$. Since $\gamma$ is diagonal, we get 
\begin{equation}
    \rho^k A_k(Uz)=\rho UA_k(z).
\end{equation}
Since $\rho\in(0,1)$, we must have $A_k=0$ for all $k\geq 2$, i.e. $A$ is a linear map.

Suppose $g_i$ are two static metrics. The above construction gives two linear maps $A_i$ on $\mathbb C^2$, such that
\begin{equation}
    \widetilde g_i=\lambda_i A_i^*g_H.
\end{equation}
This allows us to compare these metrics using a more invariant form of Hopf metric.
    To be more precise, let $Q$ be a positive-definite Hermitian form on $\mathbb C^2$. Then we define 
    \begin{equation}
        g_{Q}\vert_z(u,v)\coloneqq\frac{\Re Q(u,v)}{Q(z,z)}.
    \end{equation}
    Denote by $Q_0$ be the standard Hermitian form, and set $Q_i=A_i^* Q_0$. It follows that
    \begin{equation}
        \widetilde g_i=\lambda_i g_{Q_i}.
    \end{equation}

Both metrics are invariant under deck transformations. Let $h$ be any deck transform. Since it preserves $g_{Q_i}$, it must be linear. The Lee form of $g_{Q}$ is given by 
\begin{equation}
    \vartheta=-d\log Q(z,z).
\end{equation}
Since it is deck invariant, we can find $c_i(h)>0$, such that
\begin{equation}
    h^* Q_i=c_i(h)Q_i.
\end{equation}
By taking determinant, we see that $c_i(h)=\vert\det h\vert$.

Let $S$ be the unique positive linear map such that
\begin{equation}
    Q_2(u,v)=Q_1(Su,v).
\end{equation}
We claim that $S$ commutes with any deck transform $h$. In fact,
\begin{equation}
         Q_1(Shu,hv)=Q_2(hu, hv)=h^*Q_2(u,v)\\
         =\vert\det h\vert Q_1(Su,v).
   \end{equation}
   On the other hand, since $h^*Q_1=\vert\det h\vert Q_1$, we also have
   \begin{equation}
       Q_1(hSu, hv)=\vert\det h\vert Q_1(Su,v).
   \end{equation}
   It follows that $S$ commutes with $h$. Then by general theory, $ S^{-\frac12}$ also commutes with all deck transforms. Moreover, using $S$ is $Q_1$-self-adjoint, we obtain that $(S^{-\frac12})^* g_{Q_2}=g_{Q_1}$, and hence $S^{-\frac12}$ descends to $M$ as a biholomorphism $\Phi$ such that
\begin{equation}
    \Phi ^*g_2=\frac{\lambda_2}{\lambda_1}g_1.
\end{equation}

\subsection{Associated Vaisman metric and an alternative proof of Theorem \ref{thm--classification of complex structure}}\label{sec:vaisman}
Recall that a Hermitian metric $g$ is locally conformally K\"ahler (LCK) if
$d\omega=\vartheta\wedge\omega$ for a closed one-form $\vartheta$, and is
Vaisman if its nonzero Lee form $\vartheta$ is $g$-parallel.

The following lemma constructs explicitly the parallel $1$-form appeared in Corollary \ref{cor--rough complex structure determination}, whose proof can be also viewed as a generalization of argument used in Proposition \ref{prop:staticity}:
\begin{lemma}\label{lem--canonical parallel 1-form}
Let $(M,J,g,f)$ be a compact steady pluriclosed
soliton surface with Lee form $\theta$. Suppose $H=-*\theta\not\equiv0$. Put
$\eta:=e^{-f}\theta$. Then $\eta$ is nowhere vanishing and satisfies
\begin{equation}
    \nabla\eta=0,\qquad
 \eta(\nabla f)=0,\qquad
 \Ric_f(\eta^\sharp,\eta^{\#})=0.
\end{equation}
 \end{lemma}

\begin{proof}

It is proved in Proposition \ref{prop:staticity} that
\begin{equation}
    \delta\theta=0,\ d\theta=df\wedge\theta.
\end{equation}
As a result, we have
\begin{equation}
    d\eta=0,
 \qquad
 \delta_{-f}\eta=0.
\end{equation}
where $\delta_{-f}:=\delta-\iota_{\nabla f}$. Note that this imples that $\eta$ is harmonic with respect to the \textit{opposite} weighted measure $e^fd\mu_g$, in contrast to the construction in Section \ref{sub--BE theory}.

Nevertheless, let $\alpha$ be the $\Delta_f$-harmonic representative of
$[\eta]$, i.e. $d\alpha=0$ and
$\delta_f\alpha=0$. The discussion in Section \ref{sub--BE theory} implies that $\alpha$ is parallel and
$\Ric_f(\alpha^\sharp,\alpha^{\#})=0$.  Since $\alpha$ is parallel, we have
$\delta\alpha=-\mathrm{tr}_g(\nabla\alpha)=0$. Combined with the fact that $\delta_f\alpha=0$, one also has
$\alpha(\nabla f)=0$, and hence $\delta_{-f}\alpha=0$.

Write $\eta-\alpha=du$. Taking the adjoint with respect to the opposite weighted measure $e^f dV_g$ yields
\begin{equation}
   \int_M|du|^2e^fd\mu_g
 =\int_M u\,\delta_{-f}(\eta-\alpha)e^fd\mu_g=0,
\end{equation}
so $\eta=\alpha$. The asserted identities then follow immediately.
Finally, $H=-*\theta\not\equiv0$ implies $\eta\not\equiv0$. Since it is parallel, it is nowhere vanishing.
\end{proof}

For a non-K\"ahler steady pluriclosed soliton, we let
\begin{equation}
    \tau:=\theta/|\theta|=\eta/\vert\eta\vert,\ T:=\tau^\#.
\end{equation}
By Lemma \ref{lem--canonical parallel 1-form}, $\tau$ and
$T$ are parallel.

\begin{proposition}[The associated Vaisman metric]
\label{prop:vaisman-metric}
Let $(M,g,J,f)$ be a compact non-K\"ahler steady
pluriclosed soliton surface. Then $(M,J)$ admits a Vaisman metric.  More precisely,
\begin{equation}\label{eq:hat-omega}
 \widehat\omega
 :=\tau\wedge(J\tau^\#)^\flat
  +\frac{|\theta|}{2}\omega_\perp=\tau\wedge(J\tau^\#)^\flat
  +\frac{|\theta|}{2}\left(\omega-\tau\wedge(J\tau^\#)^\flat\right)
\end{equation}
is the fundamental form of such a metric, whose parallel Lee form is
$2\tau$.  
\end{proposition}

\begin{proof}
Since $g$ is non-K\"ahler, $H=-*\theta\not\equiv0$, so the preceding lemma
applies. Therefore we have
\begin{equation}
    T(f)=0,
 \qquad \Ric_f(T,T)=0.
\end{equation}
The first implies that $T(|\theta|)=0$, since
$|\theta|=e^f|\eta|$ and $\vert\eta\vert$ is a constant. Then using $H^2(T,T)=2|\iota_TH|^2$ and the soliton equation, 
 we obtain \(
0=\Ric_f(T,T)=\frac12|\iota_TH|^2,
\) and hence
\begin{equation}
    \iota_TH=0.
\end{equation}
Using $\nabla^B J=0$, $\iota_T H=0$, and the fact that $T$ is parallel, the same argument as in Proposition \ref{prop--holomorhic killing vector fields} yields
$T$ is real holomorphic.  In particular,
$\mathcal L_T(JT)=(\mathcal L_TJ)T=0$. Together with $\mathcal L_T|\theta|=0$,
$\mathcal L_Tg=0$ and $\mathcal L_T\tau=0$, this shows that the flow of $T$ preserves $\omega$,
$(J\tau^\#)^\flat$, and $\omega_\perp$.

By definition, we have $\iota_T\omega=(J\tau^\#)^\flat$. Since the Lee form is characterized by $d\omega=\theta\wedge\omega$, we then have $d\omega=|\theta|\tau\wedge\omega_\perp$. Cartan's formula gives
\begin{equation}\label{eq:dzeta}
 d(J\tau^\#)^\flat=-\iota_Td\omega=-|\theta|\omega_\perp.
\end{equation}
The form $\widehat \omega$ in \eqref{eq:hat-omega} is positive of type $(1,1)$:
$\tau\wedge(J\tau^\#)^\flat$ is its positive block on $\mathrm{Span}\{ T,JT\}$, while
$|\theta|/2$ rescales the positive block on the orthogonal $J$-invariant
plane. Hence it is the fundamental form of a Hermitian metric
$\widehat g$. The identities above give
\begin{align*}
 d\widehat\omega
 &=-\tau\wedge d(J\tau^\#)^\flat
   +d\left(\frac{|\theta|}{2}\omega_\perp\right)\\
 &=|\theta|\tau\wedge\omega_\perp
 =2\tau\wedge\widehat\omega.
\end{align*}
Thus $\widehat g$ is locally conformally K\"ahler with Lee form $2\tau$.
The $\widehat g$-dual of $\tau$ is still $T$, since $\widehat\omega$ only rescales the orthogonal direction. Moreover, every term in
\eqref{eq:hat-omega} is $T$-invariant, hence
$\mathcal L_T\widehat g=0$. The symmetric and alternating parts of
$\nabla^{\widehat g}\tau$ are respectively
$\frac12\mathcal L_T\widehat g$ and $\frac12d\tau$, so
$\nabla^{\widehat g}\tau=0$. Therefore $2\tau$ is $\widehat g$-parallel and nonzero. It follows that
$\widehat g$ is Vaisman.
\end{proof}

Using the existence of Vaisman metric, we can give a different proof of Theorem \ref{thm--classification of complex structure} without using Kodaira's characterization \cite[Theorem 41]{KodairaIII} or the topological characterization of K\"ahler surfaces, as follows. By Lemma \ref{lem--tensor algebra lemma} and Bochner technique, we know that $b_1(M)\leq 1$. We may assume $H\not\equiv 0$, so by Lemma \ref{lem--canonical parallel 1-form} there exists a non-zero parallel $1$-form. In particular, we must have $b_1(M)=1$. Moreover, Proposition~\ref{prop:vaisman-metric} gives a Vaisman
metric $(\widehat g,\widehat\omega)$ on $(M,J)$.  

Let $\widetilde M$ be the universal cover of $M$. The same argument in Corollary \ref{cor--rough complex structure determination} implies that there is a de Rham splitting for metric $g$, i.e. $\widetilde M=N\times\mathbb R$ for a compact simply-connected manifold $N$. Since $\ker\widetilde\tau= TN$ and $\tau$ is also $\widehat g$-parallel, the pull-back of $\widehat g$ to $\widetilde M$ also splits as product $N\times\mathbb R$.

The de Rham splitting for $\widetilde {\widehat g}$ gives
\begin{equation}
    (\widetilde M,\widetilde {\widehat g})\cong(\mathbb R\times N,dt^2+g_N),\ \widetilde\tau=dt.
\end{equation}
It follows from a similar argument for static case in Section \ref{sec--proof of complex structure classification} that $\widetilde M$ admits a K\"ahler cone metric. 
By \cite{van11}, $\widetilde M\cup\{o\}$ defines a normal affine variety. Since $\pi_1(L)=\{1\}$, by Mumford's theorem \cite{mumford61} we see $o$ is a regular point of $\widetilde M\cup\{o\}$, thus $\widetilde M\cup\{o\}$ is biholomorphic to $\mathbb C^2$ by a standard commutative algebra argument. Therefore $(M,J)$ is a Hopf surface.

Let $p\colon \widehat M\to M$ be an arbitrary primary cover. The
pullback of $\widehat\omega$ is again Vaisman, and Belgun's classification
\cite[Theorem~1]{Belgun00} then implies that the primary Hopf surface
$\widehat M$ is of class~$1$. 

\begin{remark}
 After the existence of Vaisman metric, one can rush immediately to the last step and use Belgun's result to conclude the Theorem \ref{thm--classification of complex structure}, given that we already know $(M,J)$ is a Hopf surface, by \cite{Streets19} or Corollary \ref{cor--rough complex structure determination}. We use the above argument instead, since the known relation between Vaisman metric and K\"ahler cone gives a more elementary proof of the fact that $(M,J)$ is a Hopf surface.
\end{remark}

\bibliographystyle{alpha}
\bibliography{references}

\end{document}